\documentclass[review,3p,authoryear]{elsarticle}

\usepackage{amssymb}

\usepackage{lineno}

\usepackage{color}

\usepackage[colorlinks,
linkcolor=red,
anchorcolor=blue,
citecolor=green]{hyperref}
\usepackage{subcaption}

\usepackage{makecell}

\usepackage{caption}

\usepackage{threeparttable}

\usepackage{amsmath}

\usepackage{algorithm}
\usepackage{algpseudocode}

\usepackage{multirow}

\usepackage{geometry}

\usepackage{ntheorem}

\usepackage{enumitem}

\usepackage{pdflscape}

\usepackage{placeins}

\usepackage{xcolor}

\usepackage{tabularx,booktabs}

\usepackage{siunitx}

\usepackage{graphicx}

\journal{Transportation research Part B: Methodological}

\begin{document}
\newtheorem{theorem}{Theorem}
\newtheorem{lemma}{Lemma}
\newtheorem{remark}{Remark}
\newtheorem{proposition}{Proposition}
\newtheorem{property}{Property}
\newtheorem*{proof}{Proof}
\newtheorem{example}{Example}

\newcommand{\ReviseColor}{gray}
\newcommand{\tarr}[2]{t^{\textnormal{arrive}}_{#1, #2}}
\newcommand{\platime}[1]{T_{#1}}
\newcommand{\toff}[1]{t^{\textnormal{off}}_{#1}}
\newcommand{\ttravel}[2]{t^{\textnormal{travel}}_{#1, #2}}
\newcommand{\qmax}{q^{\textnormal{max}}}
\newcommand{\gmax}[1]{g^{\textnormal{max}}_#1}
\newcommand{\firsthw}[1]{\underline{\tau}_#1}
\newcommand{\secondhw}[1]{\overline{\tau}_#1}
\newcommand{\tadjub}{\overline{t}^{\textnormal{adjust}}}
\newcommand{\tadjlb}{\underline{t}^{\textnormal{adjust}}}
\newcommand{\p}{\mathbf{p}}
\newcommand{\Lmax}[1]{L^{\textnormal{max}}_#1}
\newcommand{\Lflow}[1]{L^{\textnormal{flow}}_#1}
\newcommand{\Cmin}{\tau^{*}}
\newcommand{\delay}[1]{t^{\textnormal{delay}}_#1}
\newcommand{\moveFirstN}[1]{\underline{n}_#1}
\newcommand{\tarrLB}[2]{\underline{t}^{\textnormal{arrive}}_{#1, #2}}
\newcommand{\tarrUB}[2]{\overline{t}^{\textnormal{arrive}}_{#1, #2}}
\newcommand{\platimeLB}[1]{\underline{T}_{#1}}
\newcommand{\platimeUB}[1]{\overline{T}_{#1}}
\newcommand{\linearvarXunder}[3]{\underline{x}_{#1, #2, #3}}
\newcommand{\linearvarXover}[3]{\overline{x}_{#1, #2, #3}}
\newcommand{\linearvarYunder}[3]{\underline{y}_{#1, #2, #3}}
\newcommand{\linearvarYover}[3]{\overline{y}_{#1, #2, #3}}
\newcommand{\TotalDemand}{q^{\text{total}}}
\newcommand{\ObjSwitchCoef}{\rho}
\newcommand{\LostTime}{t^{\textnormal{lost}}}

\allowdisplaybreaks\allowdisplaybreaks

\begin{frontmatter}



	\title{Cyclic Modulation Control of Multi-Conflict Connected Automated Traffic}


	\affiliation[inst1]{organization={Department of Civil and Environmental Engineering, Texas A\&M University},
		postcode={77843},
		state={Texas},
		country={United States}}
        \affiliation[inst3]{organization={Department of Mechanical Engineering, Texas A\&M University},
		postcode={77843},
		state={Texas},
		country={United States}}

        \author[inst1]{Fan Pu}

        \author[inst1]{Zihao Li}
        
        \author[inst3]{Sivakumar Rathinam}
        
        \author[inst3]{Minghui Zheng}

	\author[inst1]{Yang Zhou\corref{cor1}}

        \cortext[cor1]{Corresponding author: Yang Zhou. Email: yangzhou295@tamu.edu}

	\begin{abstract}
        Multi-conflict traffic is ubiquitous. Connected Automated Vehicles (CAVs) offer unprecedented opportunities to enhance safety, reduce emissions, and increase throughput through precise coordination and automation. However, existing CAV strategies remain confined to specialized scenarios, such as highway on-ramp merging or single-lane roundabouts, and traditional traffic signals sacrifice efficiency for safety via rigid phasing and all-red intervals. In this paper, we present Cyclic Modulation Control of Multi-Conflict Connected Automated Traffic (CMAT), a unified, geometry-agnostic framework that embeds each conflict point into a repeating sequence of “micro-phases”. Vehicles dynamically form platoons with demand-responsive sizes and negotiate time slots for occupying conflict points, enabling collision-free traversal and high intersection utilization. CMAT aims to minimize delay, guarantee safety, and accommodate arbitrary merging, diverging, and crossing patterns without manual retuning. We formalize CMAT as a mixed-integer linear programming model constructed on a directed graph abstracted from the physical intersection layout. The performance of CMAT is evaluated across a suite of multi-conflict tests, including simple two-way crossings, four-leg intersections, complex connected intersections. The results demonstrate substantial reductions in delay and significant throughput improvements compared with state-of-the-art CAV coordination methods and traditional signal timing strategies.

	\end{abstract}



\end{frontmatter}

\section{Introduction}
The advent of Connected Automated Vehicles (CAVs) promises transformative gains in roadway safety, emissions reduction, and network efficiency by enabling cooperative, fine-grained control. Substantial progress has been made on fundamental driving tasks—including car-following~\citep{zhou2020stabilizing,ma2017parsimonious,guo2019joint, wang2015game, li2024enhancing}, lane keeping~\citep{yi20232d, liu2020lateral, jing2025efficient}, and lane changing~\citep{yi2024cooperative, duan2023cooperative, jing2025efficient, luo2016dynamic}. Yet the full potential of CAVs remains unrealized in environments where vehicle paths intersect, merge, or diverge. Such multi-conflict scenarios, common at roundabouts, ramps, weaving segments, and urban intersections, are safety-critical and frequently determine corridor and network throughput~\citep{chen2018capacity}.

Existing CAV coordination strategies generally address individual conflicts in isolation. Sequencing-based controllers developed for specific geometries such as highway ramps~\citep{li2023sequencing, chen2021connected, liu2023safety, hou2023cooperative} or single-lane roundabouts~\citep{martin2021traffic, naderi2023controlling, li2025virtual} typically rely on FIFO rules, mixed-integer programming, or virtual car-following~\citep{xu2018distributed, chen2022conflict} and space reservations~\citep{li2013modeling, levin2016optimizing, liu2023safety}. While effective locally, these controllers are tightly coupled to assumed layouts and conflict patterns~\citep{zhu2022merging}, limiting their scalability to intersections where merging, diverging, and crossing maneuvers coexist~\citep{wu2022intersection, li2022review}. Likewise, AI- and optimization-based intersection controllers~\citep{fayazi2017optimal, yu2018integrated, fayazi2018mixed, antonio2022multi, li2023intersection} improve local sequencing but provide little insight into the global flow-level principles governing stability and coordination. Their reliance on scenario-specific assumptions, training conditions, or geometric simplifications often hinders generalization across layouts and demand patterns.

A key limitation across these methods is their focus on individual vehicle trajectories rather than the collective flow behavior that emerges across interacting maneuvers. Maneuver-specific time headways differ substantially, e.g., merging requires larger headways than standard car-following, and the failure to account for such heterogeneity can reduce effective capacity and trigger oversaturation under high or imbalanced demand~\citep{talebpour2016influence, montanino2021string}. Rhythmic control~\citep{lin2021rhythmic, chen2021rhythmic} provides a valuable step toward decentralized coordination by assigning repetitive movement-specific rhythms, but its fixed patterns cannot adapt to real-time variations in arrival processes or maneuver characteristics. These challenges indicate the need for a flow-level, demand-responsive perspective that explicitly captures how conflict structures, maneuver headways, and arrival variability jointly shape network performance.

On the other hand, traditional traffic signals embody this flow-level perspective by aggregating vehicles into phases and separating conflicts in time. However, fixed phasing, conservative clearance intervals, and mandatory all-red times constrain their ability to exploit CAV connectivity and real-time cooperation~\citep{gartner1991multi, gallivan1988optimising, osorio2015urban, lu2022autonomous}. Signal-bound CAV trajectory planning~\citep{yao2020decentralized, wang2022connected, guo2019joint, yao2023rolling, wu2021cooperative} offers incremental improvements but remains confined to predefined phases and cycle lengths. The growing tension between microscopic CAV capability and macroscopic signal rigidity underscores the need for a unified control paradigm that retains the scalability of flow aggregation while harnessing the adaptability of continuous CAV coordination.

To address these gaps, this paper introduces Cyclic Modulation Control of Multi-Conflict Connected Automated Traffic (CMAT), a unified, geometry-agnostic framework that synthesizes flow-level coordination with CAV-enabled real-time operation. CMAT replaces fixed signal phases with demand-responsive conflict cycles in which CAVs are grouped into platoons whose sizes adapt to prevailing arrivals. Within each unified cycle, right-of-way is allocated among competing movements through synchronized platoon-level slot negotiation, ensuring safe headway separation without the capacity loss of all-red intervals. By abstracting arbitrary conflict structures into a directed graph, CMAT accommodates diverse merging, diverging, and crossing geometries without requiring grid alignment, convexification of road boundaries, or other geometry-dependent adaptations.


CMAT offers three key advantages. (1) Demand responsive: Platoon sizes adjust automatically to movement-specific and time-varying arrival patterns, enabling capacity-maximizing allocation of right-of-way.
(2) Geometry and conflict adaptive: A general conflict-graph representation permits application to arbitrary layouts without geometric modification or restrictive assumptions on lane alignment and shape.
(3) Scalable and extensible: By coordinating interactions at the platoon level, CMAT converts microscopic trajectory scheduling into mesoscopic flow coordination, offering scalability and straightforward deployment across intersections and transportation networks of varying size and complexity.

The remainder of this paper is organized as follows. \autoref{sec:problem} presents a motivating example for CMAT. \autoref{sec:math} describes the problem and introduces the mathematical formulation. \autoref{sec:exps} reports numerical simulations and analyzes traffic-flow implications. \autoref{sec:conc} concludes with findings and directions for future research.

\section{Motivating Example}\label{sec:problem}
Managing fully automated traffic at general intersections fundamentally involves resolving right-of-way conflicts arising from spatial and temporal overlaps among vehicle trajectories. Achieving desired service levels, such as a target total throughput, adds further complexity. The proposed CMAT framework addresses these challenges through a design grounded in traffic flow theory.

For a given intersection, the achievable throughput $q$ depends on how frequently vehicles can safely enter the conflict area and the speed $v$ at which they pass through:
\begin{equation}
    q = \frac{1}{\alpha\tau_c + (1 - \alpha)\tau_f + l/v},
\end{equation}
where $\alpha \in [0,1]$ denotes the fraction of time vehicles encounter crossing conflicts, and $l$ is the average vehicle length. The parameters $\tau_c$ and $\tau_f$ are the upper and lower bounds of admissible CAV time headways: $\tau_c$ corresponds to the safety headway required for crossing conflicts, while $\tau_f$ represents car-following headway at free-flow speed $v_f$. Since crossing collisions are more severe than car-following interactions, it follows that $\tau_{c} > \tau_{f}$, and the average headway is a convex combination of the two.

Two design principles follow directly from this relationship. First, as $\alpha$ grows, indicating more frequent crossing interactions, the effective headway $\alpha\tau_c + (1-\alpha)\tau_f$ increases and throughput $q$ declines. This provides a clear mathematical rationale for platooning: grouping vehicles reduces the frequency of crossing conflict events (i.e., lower $\alpha$), allowing most vehicles to operate under the smaller car-following gap $\tau_f$ rather than the larger crossing conflict gap $\tau_c$. Second, throughput increases with travel speed and reaches its maximum when $v=v_f$.

Combining these insights motivates the design of a ``phantom'' micro-signal mechanism: instead of a traditional red-green phase structure, each movement intermittently forms small platoons from random arrivals, releases platoons to pass the intersection following the free-flow speed. By doing so, the system sustains a lower average headway while eliminating unnecessary deceleration cycles. This achieves both efficient headway structure and intuitively smoother, predictable, and conflict-aware coordination.

\autoref{fig:motiv-single} illustrates the operational benefits of ``phantom'' micro-signal in a single-conflict setting. The left and right sides show the traditional stop-sign control and the proposed CMAT strategy, respectively. It is evident that CMAT improves intersection efficiency by reducing costly crossing-conflict negotiations and making car-following headways the dominant interaction mode. As a result, intersections can accommodate substantially higher traffic throughput, an effect that becomes particularly valuable under high-demand conditions.

\begin{figure}[htbp]
    \centering
    \includegraphics[width=0.6\linewidth]{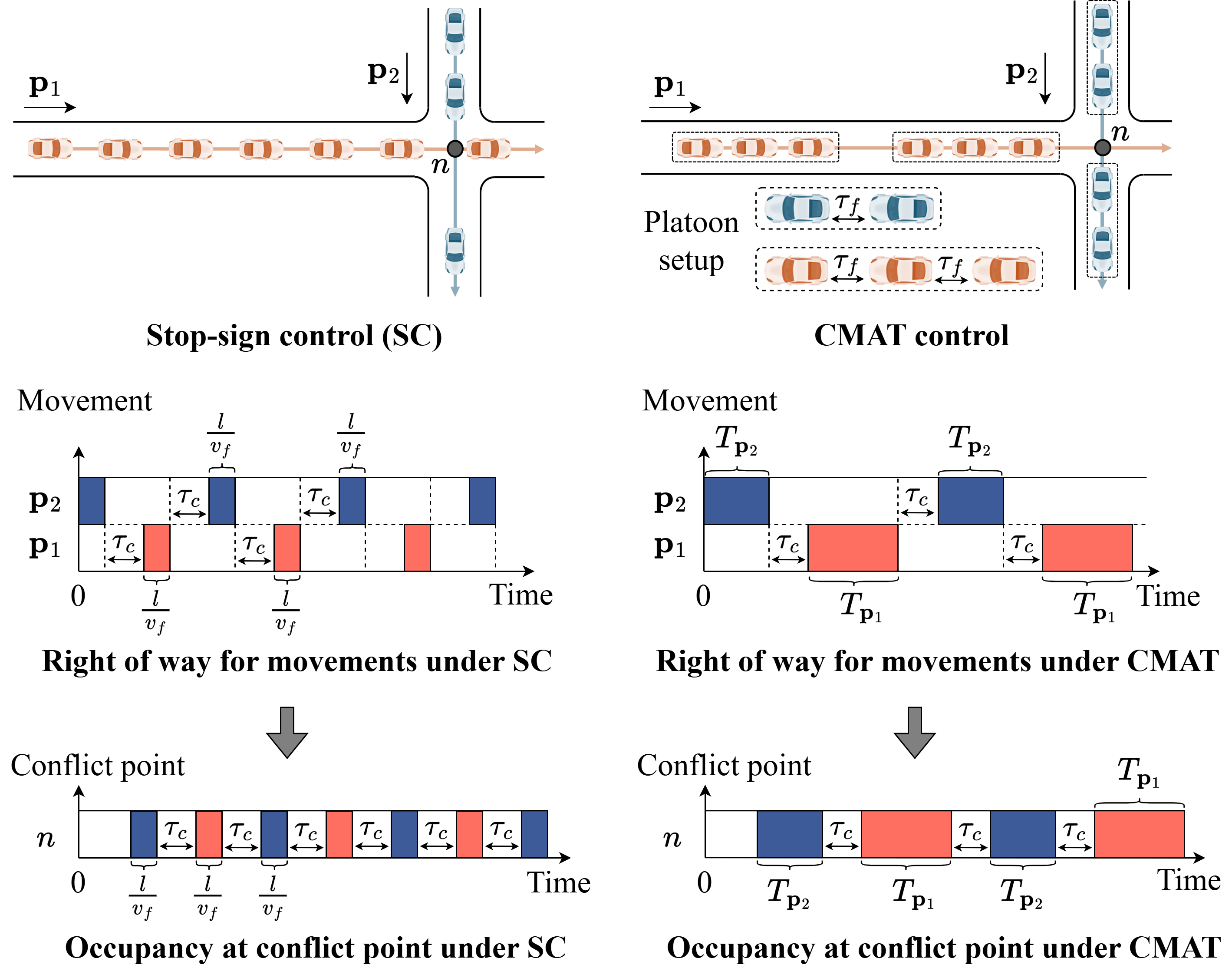}
    \caption{Comparison of control strategies in the single-conflict scenario}
    \label{fig:motiv-single}
\end{figure}

To further quantify the outcome, we consider the same single-conflict scenario. Under platooning, only the leading vehicle of each platoon follows the crossing headway $\tau_f$. If both movements form platoons with average size $N$, the effective conflict frequency $\alpha = \frac{2}{2+2(N-1)} = \frac{1}{N}$, since $2N$ pairs of vehicle interactions happen at the conflict point with two crossing headways incurred by leading vehicles and remaining $2(N-1)$ intra-platoon car-following interactions. Therefore,
\begin{equation}\label{eq:moti-q}
    q(\alpha)=\frac{1}{\alpha\tau_c + (1 - \alpha)\tau_f + l/v_f},\qquad q(N) = \frac{1}{\tau_f + \frac{1}{N}(\tau_c-\tau_f) + l/v_f},
\end{equation}
which increases monotonically in $N$ and is bounded above by $1/(\tau_f + l/v_f)$ as $N \to \infty$ or $\alpha\to 0$. 

The blue curve in \autoref{fig:moti-func}(a) shows the maximum capacity $q(\alpha)$ for different values of $\alpha$. The upper and lower bounds, $\overline{q}$ and $\underline{q}$, represent the ideal and worst-case conditions in which all vehicles follow $\tau_f$ or $\tau_c$, respectively. The area between $q(\alpha)$ and $\underline{q}$ represents the full range of throughput levels that can be achieved for a given $\alpha$ configuration. For any average demand $\TotalDemand$, points that fall within the shaded region indicate feasible operating conditions that can fully serve that demand, which is the operating range CMAT is designed to achieve.

\begin{figure}[htbp]
    \centering
    \includegraphics[width=1\linewidth]{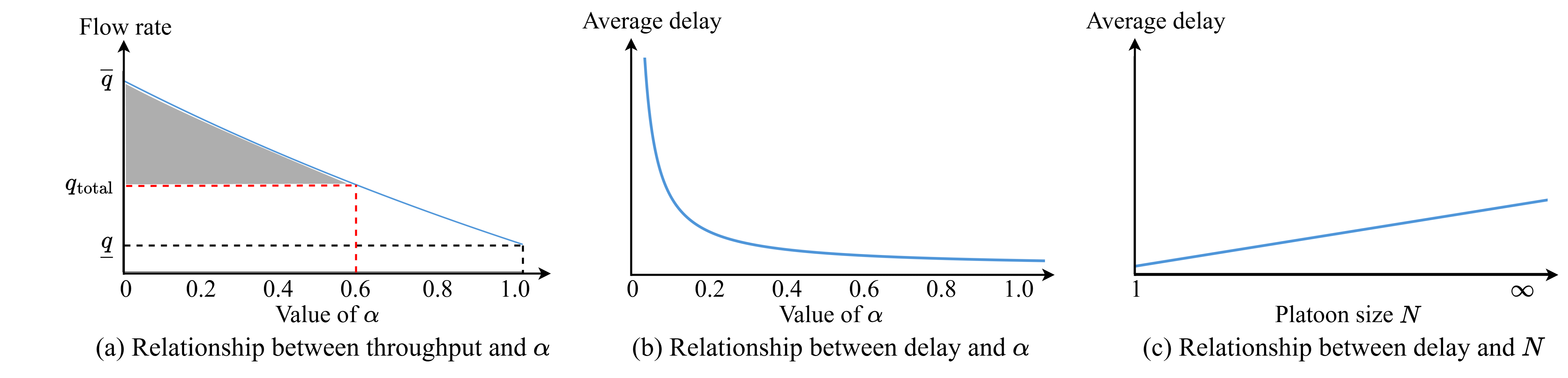}
    \caption{Illustration of functions used in the motivating example}
    \label{fig:moti-func}
    \captionsetup{justification=centering}
\end{figure}

We next examine the delay implications of platooning under the uniform arrival condition. Let $\{X_k\}_{k\ge 1}$ be the time intervals between the ($k-1$)-th and $k$-th arrivals on the same movement, which are independent and identically distributed (i.i.d.) with a finite mean $\mu:=\mathbb{E}[X_1]$. Define the arrival epochs $\{T_N\}$ in the usual renewal way:
\begin{equation}
    T_0=0,\quad T_N= \sum_{k=1}^{N} X_k, \quad N\ge 1
\end{equation}

So $T_N$ is the time of the $N$-th vehicle's arrival, which is also the time to accumulate a platoon of size $N$. If the platoon is immediately released at $T_N$, the $i$-th vehicle ($i=1,\dots,N$) experienced waiting time
\begin{equation}
    W_i = T_N - T_{i} = \sum_{k=i+1}^{N} X_k
\end{equation}

Therefore, the expected waiting time of the $i$-th vehicle is
\begin{equation}
    \mathbb{E}[W_i] = \sum_{k=i+1}^{N} \mathbb{E}[X_k] = (N - i)\mu,
\end{equation}
and the average delay per vehicle in the platoon is
\begin{equation}
    \bar{W}(N) = \frac{1}{N}\sum_{i=1}^{N} \mathbb{E}[W_i]=\frac{N-1}{2} \mu, \qquad \bar{W}(\alpha)=\frac{1-\alpha}{2\alpha}\mu,
\end{equation}
which as shown in \autoref{fig:moti-func}(b) and \autoref{fig:moti-func}(c), the average delay $\bar{W}(N)$ grows linearly when increasing $N$.

In summary, when $\alpha$ decrease (or N increase), the capacity is improved via smaller effective headway but also increases average waiting time linearly in $N$. CMAT should adaptively choose $N$ (equivalently $\alpha$) to find the ideal trade-off between the capacity gain in $q(N)$ and the increasing delay $\bar{W}(N)$ (i.e., $q(N)$ satisfy the traffic demand with minimum $N$).

Based on the above analysis, extending CMAT to more general multi-conflict intersections introduces new challenges. While platooning can increase throughput and interactions at free-flow speed can reduce delay locally, these benefits may diminish without coordination across all conflict points. Two interdependent design components therefore become essential:
(i) \textit{Demand grouping}: forming platoons of appropriate size on each movement so that arrival demand is met;
(ii) \textit{Platoon sequencing}: determining, at each conflict point, the passing order and offset (timing) of platoons from competing movements.

\begin{figure}[htbp]
\centering
\includegraphics[width=1\linewidth]{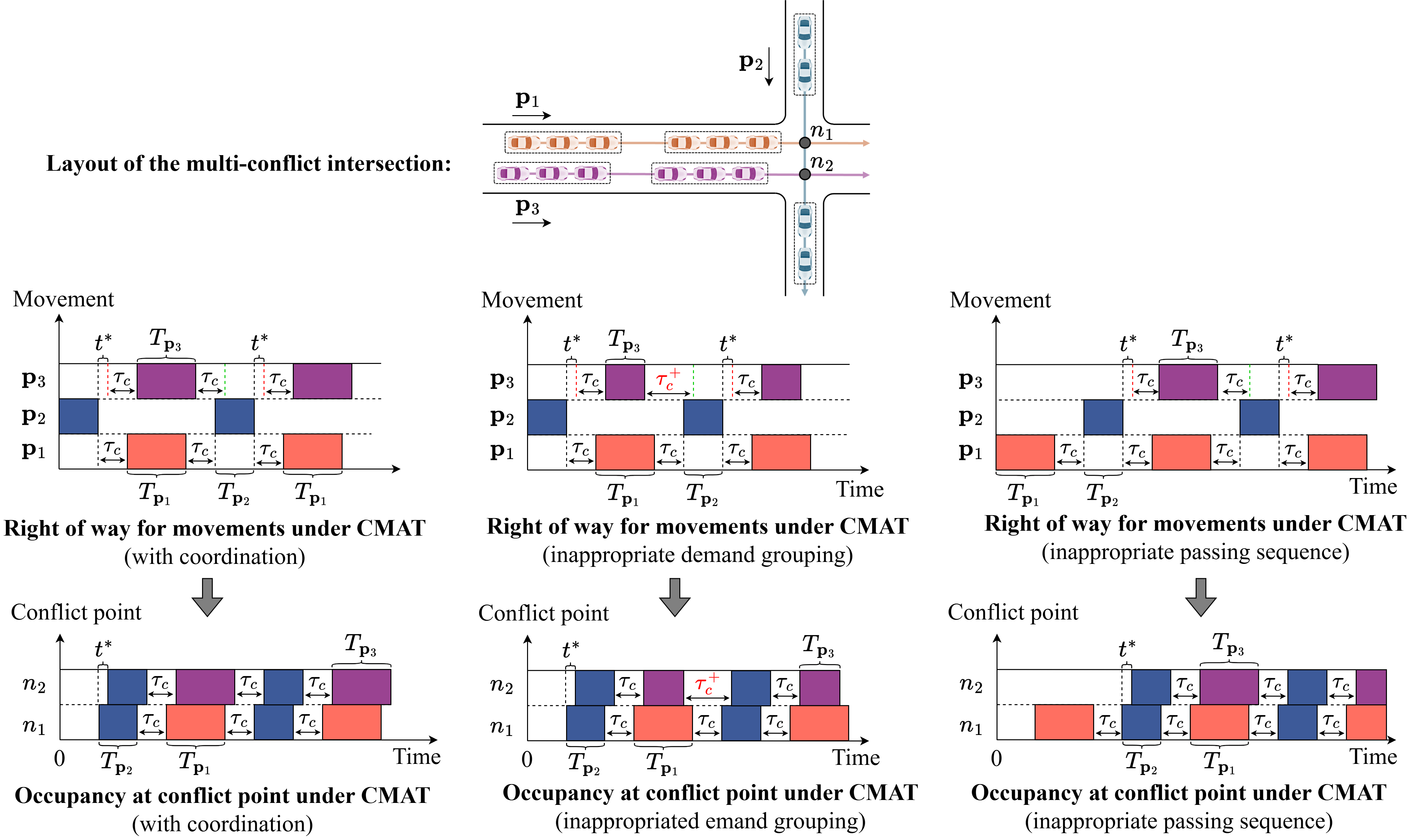}
\caption{Solution comparison in the multi-conflict scenario.}
\label{fig:motiv-multi-solution}
\end{figure}

To illustrate the importance, \autoref{fig:motiv-multi-solution} shows a three-movement intersection with two conflict points. Due to the geometry offset, let $t^*$ denote the travel time for a vehicle to move from $n_1$ to $n_2$ at speed $v_f$. In the left panel, properly coordinated grouping and sequencing yield more efficient operation that improves the usage of both conflict points. The middle panel highlights the capacity loss caused by improper demand grouping, where mismatched platoon sizes lead to excess headway $\tau_c^+>\tau_c$. In the right panel, poor sequencing causes a substantial delay for the platoon on movement $\p_3$, since it must wait for the platoon on movement $\p_2$ to clear point $n_2$.

This analysis highlights the need for a holistic framework that accounts for both local conflict resolution and system-level coordination. CMAT addresses this by jointly optimizing platoon sizing, platoon sequencing, and time offsets for the ``phantom'' micro-signals across the entire intersection. This ensures conflict-free operation while maintaining an effective balance between throughput improvement and delay performance.

\section{Problem description and mathematical formulation}\label{sec:math}
In this section, we formally describe CMAT and present a mixed-integer linear formulation. CMAT is a geometry-adaptive, conflict-resolution, and demand-responsive approach to jointly platooning CAVs and coordinating their right-of-ways within intersections. The main assumptions of CMAT are as follows:
\begin{itemize}
    \item origin-destination is fully known
    \item a fully CAV environment is required for precise trajectory regulation;
    \item vehicles within the control zone travel at a unified constant cruising speed;
    \item lane within control area are prohibited to eliminate additional conflicts;
    \item shared through–left-turn lanes are not allowed, which simplifies conflict resolution.
\end{itemize}

CMAT is designed to adapt to general intersections without imposing geometric constraints by abstracting the physical intersection layout into a directed graph, commonly adopted in network flow problems \citep{yin2023integrated,pu5019798optimal}, where the essential geometric characteristics are preserved: nodes represent conflict points, while arcs denote the connections between these points and preserve the corresponding lengths of road segments.
To introduce the basic concepts and terminology employed in CMAT, we use a general four-leg intersection shown in \autoref{fig:prob_desc} as an example. Within the intersection, a \emph{conflict point} refers to a physical location where the trajectories of two or more vehicle movements intersect, merge, or diverge. These points represent potential collision locations in the absence of proper right-of-way control.
In CMAT, the \emph{control area} is a region within the intersection that encompasses all conflict points. Within the control area, a traffic management system actively coordinates CAV movements to ensure safe and efficient passage through potential conflicts.

The computational complexity of intersection control largely stems from the detailed vehicle dynamics that occur within the control area. To reduce this burden, many existing studies simplify the problem by disregarding acceleration and deceleration behaviors and instead assume that vehicles traverse the intersection at a constant speed \citep{au2015autonomous, chen2021rhythmic, wu2022autonomous, niels2024optimization}. This assumption not only creates a more predictable operating environment but also increases capacity, as higher operating speeds directly improve throughput (see \autoref{eq:moti-q}).

To support this constant-speed assumption, CMAT requires all CAVs to enter the control area at the free-flow speed $v_f$. This is achieved through an upstream adjustment (acceleration) zone, typically about 100 m in length \citep{yu2019managing}, where vehicles adjust to $v_f$ before entering the controlled region. The adjustment zone plays three crucial roles: (i) speed synchronization reduces crash risk by eliminating large speed differentials near the intersection, (ii) operating at $v_f$ increases achievable discharge flow and thus raises intersection throughput, and (iii) enforcing a homogeneous speed profile suppresses disturbance formation, thereby improving the robustness and stability of the overall system.

In addition, rather than imposing a tight operational bottleneck directly at the conflict points, where geometric constraints are most restrictive, the adjustment zone effectively shifts part of this bottleneck upstream. This spatial redistribution alleviates pressure at the conflict points and enables the system to balance throughput limitations more flexibly across space. For modeling simplicity, consistent with \cite{chen2021rhythmic}, the travel time within the adjustment zone is not explicitly included in the scheduling formulation. Since platooning has been shown to significantly improve capacity, CMAT equips the entrance of the adjustment area with a ``phantom'' micro-signal that assigns platoon sizes for each origin–destination movement based on arrival demands. This mechanism enables the control zone to operate at the platoon level rather than managing interactions among individual vehicles.

\begin{figure}[tbp]
    \centering
    \includegraphics[width=0.55\linewidth]{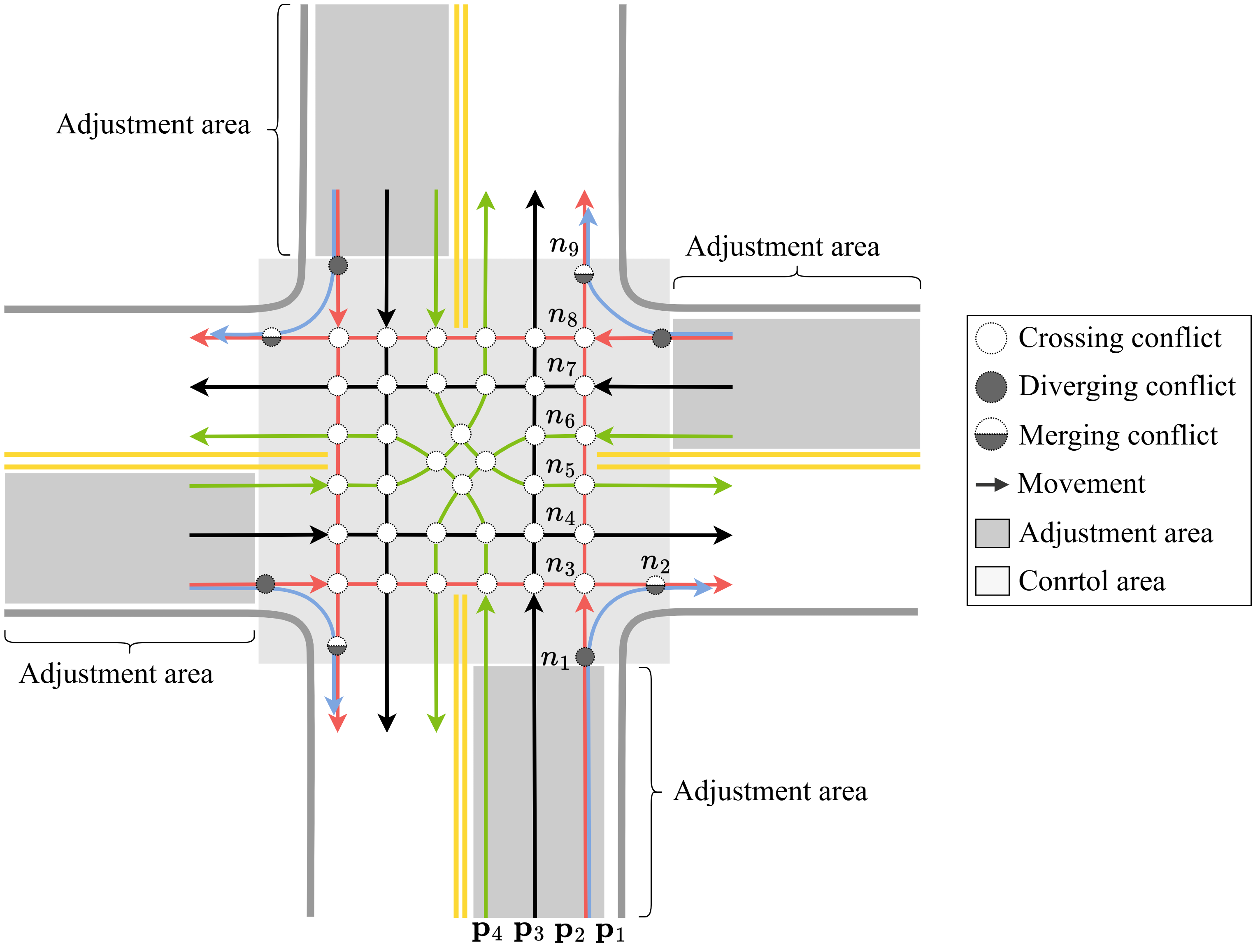}
    \caption{Illustration of the CMAT problem for a four-leg intersection (colored arrows indicate different movements.)}
    \label{fig:prob_desc}
    \captionsetup{justification=centering}
\end{figure}

We now present the mathematical formulation of CMAT. Let $\mathbf{G} = (\mathbf{N}, \mathbf{A})$ denote a directed graph representing the control area, where $\mathbf{N}$ is the set of conflict points and $\mathbf{A}$ is the set of directed arcs that connect adjacent conflict points, with each arc oriented in the direction of vehicle movement.
We define a predetermined set of origin-destination (OD) movements $\mathbf{P}= \{\p_1,\p_2, \dots, \p_m\}$, where each OD-movement $\p\in\mathbf{P}$, composed of a sequence of conflict points, corresponds to the vehicle trajectory of a specific OD. For instance, the intersection illustrated in \autoref{fig:prob_desc} comprises 16 movements in total, with four movements for each approach. As an example, the right-turn movement $\p_1$ and through movement $\p_2$ are defined as $\{n_1,n_2\}$ and $\{n_1, n_3,n_4,\cdots,n_8,n_9\}$, respectively.

The formulation of CMAT includes the following input parameters:
\begin{itemize}[itemsep=0pt, parsep=0pt, topsep=0pt, partopsep=0pt]
\item $v_{f}$: free-flow speed (m/s);
\item $l$: length of a single vehicle (m);
\item $\tau_{f}$: rear-to-front bumper car-following headway (s) at speed $v_f$;
\item $\tau_c$: minimum rear-to-front bumper headway (s) for merge, diverge, and crossing conflicts;
\item $q^{\max}$: maximum (saturation) flow rate (veh/s), computed as $\frac{1}{\tau_{f} + l / v_{f}}$;
\item $q_{\p}$: flow rate (veh/s) corresponding to movement $\p$;
\item $\ttravel{\p}{n}$: travel time (s) from the first conflict point of movement $\p$ to node $n \in \p$ at free-flow speed $v_f$.
\item $\mathbf{P}_n$: set of movements that share the conflict point $n$.
\item $\Cmin$: the threshold vehicle arrival headway (s) to check whether the flow rate is sufficiently large.
\end{itemize}

Since all CAVs within the control zone maintain the free-flow speed $v_f$, the travel time $\ttravel{\p}{n}$ can be directly computed as the length of the corresponding arc divided by $v_f$. Because shared through/right-turn lanes are not considered, each conflict point is associated with exactly two distinct movements, that is, $|\mathbf{P}_n| = 2$. This assumption simplifies platoon sequencing at each conflict point to a binary choice between the two movements.

\begin{table}[htbp]
\centering
\small
\caption{Decision variables (``Cont.'', ``Bin.'' and ``Int.'' denote continuous, binary, and integer variables, respectively)}
\label{tab:decision-vars}
\begin{tabular}{llp{13cm}}
\hline
Symbol & Type & Description \\ \hline
$C$ & Cont. & The unified cycle length (s) \\
$r_{\p}$ & Cont. & Red time (s) for movement $\p$ in one cycle \\
$g_{\p}$ & Cont. & Green time (s) to clear queue for movement $\p$ in one cycle at rate $\qmax$ \\
$L_{\p}$ & Int. & Platoon size of movement $\p$ \\
$\toff{\p}$ & Cont. & Time offset to start the first cycle on movement $\p$ \\
$\platime{\p}$ & Cont. & Duration (s) a platoon on movement $\p$ occupies a conflict point while traveling at $v_f$ \\
$\firsthw{n}$ & Cont. & The first departure-arrival headway (s) between platoons at conflict point $n$ in one cycle \\
$\secondhw{n}$ & Cont. & The second departure-arrival headway (s) between platoons at conflict point $n$ in one cycle \\
$\tarr{\p}{n}$ & Cont. & Time slot at which the platoon of movement $\p$ initially arrives at conflict point $n$ \\
$z_{\p_1,\p_2,n}$ & Bin. & For movements $\{\p_1, \p_2\}=\mathbf{P}_{n}$ at conflict point $n$, equals 1 if CAVs on $\p_1$ passes $n$ before those on $\p_2$ within the cycle, and 0 otherwise \\
\hline
\end{tabular}
\end{table}

We summarize the decision variables in \autoref{tab:decision-vars}. To aid interpretation, \autoref{fig:decision_vars} visually illustrates the physical meaning of these variables using a staggered T-intersection with two conflict points $n_1$ and $n_2$. For demonstration purposes, the platoon sizes in \autoref{fig:decision_vars} are pre-given as $L_{\p_1} = 2$, $L_{\p_2} = 3$, and $L_{\p_3} = 2$. The platoons on movements $\p_2$ and $\p_3$ are sequenced to pass through conflict points before those on movement $\p_1$; that is, $z_{\p_2,\p_1,n_1} = 1$ and $z_{\p_3,\p_1,n_2} = 1$.

\begin{figure}[htbp]
    \centering
    \includegraphics[width=0.95\linewidth]{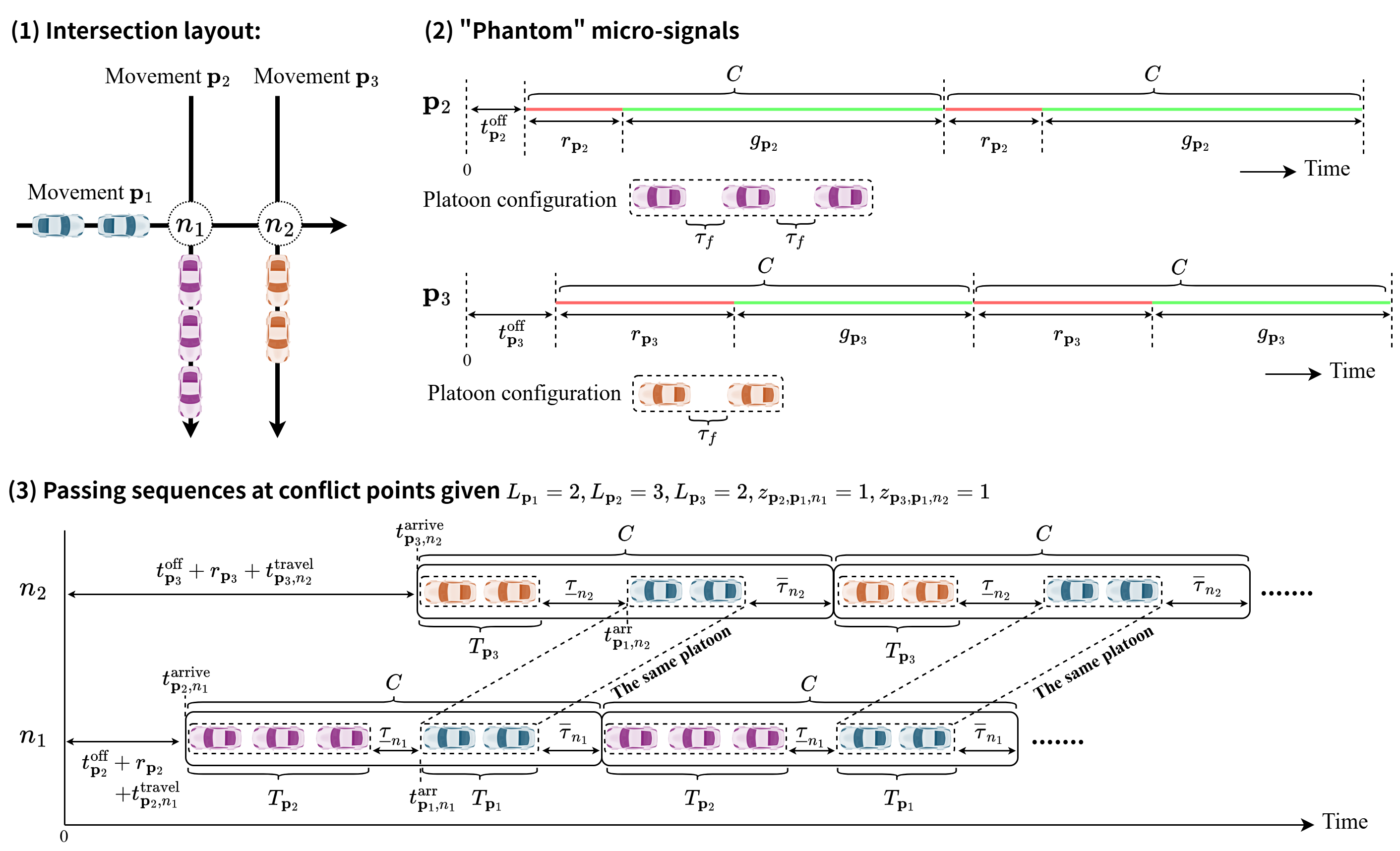}
    \caption{Illustration of decision variables}
    \label{fig:decision_vars}
    \captionsetup{justification=centering}
\end{figure}

As shown in \autoref{fig:decision_vars}, each movement equips a ``phantom'' micro-signal that operates on a unified cycle $C$, for the purpose of traffic coordination in a repetitive manner, consisting of a red interval $r_\p$ and a green interval $g_\p$. The red interval is used to accumulate a desired number of CAVs to form a platoon, while the green interval allows the platoon to discharge at the maximum rate $\qmax$. Therefore, the platoon size is determined by $L_\p=\qmax g_\p$. Each platoon then travels at the free-flow speed $v_f$ and maintains an intra-platoon car-following headway of $\tau_f$ until it exits the control zone. Moreover, to prevent potential conflicts, each movement is assigned a time offset $\toff{\p}$ that delays the opening of its ``phantom'' micro-signal.

Because CAV platoons travel at a constant speed $v_f$ within the control area, their arrival and occupancy times at each conflict point can be explicitly determined from the decision variables. Since vehicles are released immediately after the red interval ends, the arrival time of the first platoon from movement $\p$ at conflict point $n$ is $\tarr{\p}{n}=\toff{\p} + r_\p + \ttravel{\p}{n}$. Given the platoon size $L_\p$, the total time that the platoon occupies a conflict point is denoted by $T_\p$. Because the cycle repeats over time and each conflict point is alternately occupied by two successive platoons from different movements, we use $\firsthw{n}$ and $\secondhw{n}$ to represent the temporal offsets between the departure of one platoon and the arrival of the next at conflict point $n$.

Based on the above, the mixed-integer formulation for CMAT is stated as follows, where $\lambda\in[0.5,1]$ is a given weighted factor.
It is also worth noting that, although the illustrative examples in \autoref{fig:prob_desc} and \autoref{fig:decision_vars} depict isolated intersections, the proposed formulation is not limited to such settings. Rather, it naturally extends to multi-connected intersections, as the model regulates CAV traffic in a movement-based manner and therefore does not rely on any physical separation between intersections. Consequently, for any given set of estimated OD-pair demands and a representation of the intersection layout as a directed graph $\mathbf{G}$, the proposed formulation is valid for solving the CMAT problem. Numerical experiments are presented later to further validate the effectiveness of the CMAT formulation in multi-connected intersection scenarios.
\begin{subequations}\label{m:min_cycle}
{
\begin{align}
    [\textbf{M1}]~ &\min~ f_1=\lambda C + (1-\lambda)\sum_{\p\in\mathbf{P}} - L_\p \label{obj:f1} \\
    \text{s.t.}
    &\quad C = \firsthw{n} + \secondhw{n} + \platime{\p_1} + \platime{\p_2}, & \forall n\in\mathbf{N},~\p_1,\p_2\in\mathbf{P}_n:\p_1\neq \p_2 \label{c:c-length-conf} \\
    &\quad C = r_{\p} + g_{\p}, & \forall \p\in\mathbf{P} \label{c:c-length-move} \\
    &\quad
    \begin{aligned}\label{c:platoon-size}
        L_\p = 
        \left\{
        \begin{aligned}
           & \qmax g_{\p},  && \text{if } 1/q_\p \le \Cmin \\
           & 1,  && \text{else}
        \end{aligned} \right.
    \end{aligned}, & \forall \p\in\mathbf{P} \\
    &\quad q_{\p}(r_{\p} + g_{\p}) - L_\p = 0, & \forall \p\in\mathbf{P},~\text{if}~1/q_\p \le \Cmin \label{c:serv-rate} \\
    &\quad \platime{\p} = (L_\p -1) \tau_f + L_\p\frac{l}{v_f}, & \forall \p\in\mathbf{P} \label{c:platoon-time} \\
    &\quad \tarr{\p}{n} = \toff{\p} + r_\p + \ttravel{\p}{n}, & \forall \mathbf{\p}\in\mathbf{P},~n\in\mathbf{\p} \label{c:arrive-time} \\
    &\quad
    \begin{aligned}\label{c:headway}
        \firsthw{n}=
        \left\{
        \begin{aligned}
           & \tarr{\p_2}{n} - \tarr{\p_1}{n} - \platime{\p_1},  && \text{if } z_{\p_1,\p_2,n} = 1 \\
           & \tarr{\p_1}{n} - \tarr{\p_2}{n} - \platime{\p_2},  && \text{else}
        \end{aligned} \right.
    \end{aligned}, & \forall n\in\mathbf{N},~\p_1,\p_2\in\mathbf{P}_n:\p_1\neq \p_2 \\
    &\quad \tarr{\p}{n}\ge 0, & \forall \p\in\mathbf{P},~n\in \p \label{c:bound1} \\
    &\quad \firsthw{n}, \secondhw{n} \ge\tau_c, \quad 0\le C \le \overline{C}, \quad z_{\p_1,\p_2,n}\in\{0,~1\} \label{c:bound2} \\
    &\quad r_\p, g_{\p}, T_\p \ge 0, \quad L_{\p}\in\mathbb{Z}^+, & \forall \p\in\mathbf{P} \label{c:bound3}
\end{align}
}
\end{subequations}

The objective function (\ref{obj:f1}) is designed to simultaneously minimize the cycle length (as a surrogate for total delay) and maximize the total throughput (the sum of platoon sizes).
Constraints (\ref{c:c-length-conf}) ensure that cycle length $C$ equals the sum of two headway intervals together with the total occupancy times of all traversing platoons. Meanwhile, constraints (\ref{c:c-length-move}) align the cycle length of the ``phantom'' micro-signal with one red and one green interval.

Constraints (\ref{c:platoon-size}) determine the platoon size for each movement. The condition $1/q_\p\le\Cmin$ indicates that the headway $1/q_\p$ between consecutive arrivals on movement $\p$ does not exceed the threshold $\Cmin$. When this condition is violated, the demand on movement $\p$ is effectively negligible because its flow rate is too low to accumulate even a single vehicle within the cycle length $\Cmin$. In such cases, enforcing $L_p=\qmax g_\p$ with integer-valued platoon sizes would force $g_p$ (and thus the cycle length) to become unreasonably large. To avoid this, the model disregards the arrival demand of movement $\p$ (we say that movement $\p$ is \emph{muted}) and instead assumes that a hypothetical CAV ($L_\p=1$) is released every cycle. This reservation of one time slot per cycle allows a single CAV to be discharged without introducing conflicts or inflating the cycle length.

Constraints \eqref{c:serv-rate} guarantee that any queue formed during the red interval is fully cleared by the end of the cycle.
Constraints (\ref{c:platoon-time}) define the time interval that platoon on movement $\p$ occupies each conflict point. Constraints (\ref{c:arrive-time}) determine the time that the platoon on movement $\p$ arrives at conflict point $n$. Constraints (\ref{c:headway}) link the departure–arrival headway at conflict point $n$ to the sequencing of the conflicting movements.
The bounds of the decision variables are specified in constraints (\ref{c:bound1})–(\ref{c:bound3}). The platoon size $L_\p$ is restricted to positive integers. To maintain safety, $\firsthw{n}$ and $\secondhw{n}$ must satisfy the minimum safety threshold $\tau_c$.

The following proposition provides theoretical justification for \emph{muting} low-demand movements and derives a lower bound on the cycle length:
\begin{proposition}\label{prop:multiple}
    Cycle length $C$ is a common multiple of the arrival headway for each unmuted movement, that is, $C=k_\p/q_\p$ for all $\p\in\mathbf{P}:1/q_\p\le\Cmin$, where $k_\p\in\mathbb{Z}^+$ is a positive integer.
\end{proposition}

\begin{proof}
    By combining constraints (\ref{c:platoon-size}) and (\ref{c:bound3}), we obtain that $L_\p=\qmax g_\p\in\mathbb{Z}^+$ for all $\p\in\mathbf{P}:1/q_\p\le\Cmin$. Furthermore, applying constraints (\ref{c:c-length-move}) and (\ref{c:serv-rate}) yields $d_\p=q_\p C-L_\p=0$ for all $\p\in\mathbf{P}:1/q_\p\le\Cmin$. This implies $q_\p C\in\mathbb{Z}^+$, since $L_\p\in\mathbb{Z}^+$. Let $q_\p C = k_\p\in\mathbb{Z}^+$ and given that $q_\p>0$, we can conclude that
    \begin{equation}
        C=k_\p/q_\p, \quad \forall \p\in\mathbf{P}:1/q_\p\le\Cmin,
    \end{equation}
    where $k_\p\in\mathbb{Z}^+$. \qed
\end{proof}

On the one hand, the above proposition indicates that if a movement $\p$ with an extremely low flow rate $q_\p$ is not muted, the cycle length must satisfy $C\ge 1/q_\p$. This requirement can render model M1 infeasible when the minimum necessary cycle length exceeds the allowable maximum, i.e., $1/q_\p > \overline{C}$. On the other hand, the cycle length $C$ is also lower bounded by the least common multiple of $1/q_\p$ across all unmuted movements $\p\in\mathbf{P}$.

Nevertheless, despite muting all such movements, the following example demonstrates that model M1 can still become infeasible:
\begin{remark}
    Model \textnormal{M1} may become infeasible under over-saturated traffic conditions. To see this, consider a simplified scenario involving a single conflict point with two through movements: south-to-north ($\p_1$) and west-to-east ($\p_2$). Suppose both movements operate at the saturation flow rate, i.e., $q_{\p_1} = q_{\p_2} = \qmax$.
    For movement $\p_1$, constraints (\ref{c:c-length-move}), (\ref{c:serv-rate}), and (\ref{c:bound3}) collectively imply that $r_{\p_1} = 0$ and $C = g_{\p_1}$, indicating that $\p_1$ holds the right-of-way when repeating the cycle. Analogously, if these same constraints are satisfied for movement $\p_2$, it must also maintain uninterrupted right-of-way, with $r_{\p_2} = 0$ and $C = g_{\p_2}$.
    However, this leads to a contradiction at the conflict point: both movements simultaneously demand continuous access, which precludes the separation of feasible arrival and departure headways as mandated by constraint (\ref{c:c-length-conf}). This counterexample thus demonstrates the infeasibility of model \textnormal{M1} under such conditions.
\end{remark}

The infeasibility described above arises from the strict requirement $q_{\p}(r_{\p} + g_{\p}) - L_\p = 0$, which enforces complete queue discharge under oversaturated traffic conditions. To restore feasibility in such cases, model M1 is reformulated into the following relaxed version by dropping constraints (\ref{c:serv-rate}):
\begin{subequations}\label{m:min_cycle}
{
\begin{align}
    [\textbf{M2}]~ &\min~ f_2= (1-\lambda) C + \lambda \sum_{\p\in\mathbf{P}} - L_\p \label{obj:f2} \\
    \text{s.t.}
    &\quad (\ref{c:c-length-conf})-(\ref{c:platoon-size}),~(\ref{c:platoon-time})-(\ref{c:bound3})
\end{align}
}
\end{subequations}

In contrast, the objective function (\ref{obj:f2}) prioritizes maximizing the total throughput under over-saturated conditions.
\begin{proposition}\label{prop:M2-feasible}
    Model \textnormal{M2} always admits a feasible solution if $\overline{C}\ge \max\{2\tau_c + 2l/v_f,~\max_{\p\in\mathbf{P}} 1/q_\p\}$ and $\Cmin\le \max\{2\tau_c + 2l/v_f,~\max_{\p\in\mathbf{P}} 1/q_\p\}-1/\qmax$.
\end{proposition}

\begin{proof}
    We prove the result by constructing a feasible solution to constraints (\ref{c:c-length-conf})-(\ref{c:platoon-size}) and  (\ref{c:platoon-time})-(\ref{c:bound3}). For every movement $\p\in\mathbf{P}$, set:
    \begin{equation}
        g_\p := \frac{1}{q_{\max}}, \qquad L_\p := 1.
    \end{equation}

    Constraints \eqref{c:platoon-size} hold automatically: for muted movements, $L_\p = 1$ by definition; and for unmuted movements, $L_\p = q_{\max} g_\p = 1$. Then, using \eqref{c:platoon-time}, we have:
    \begin{equation}
        \platime{\p} = (L_\p - 1)\tau_f + L_\p\frac{l}{v_f} = \frac{l}{v_f}.
    \end{equation}

    Let $r_\p := C - g_\p$ for each $\p\in\mathbf{P}$. Constraint \eqref{c:c-length-move} is satisfied once $C$ is chosen. Without loss of generality, let $\mathbf{P}^*$ be an ordered sequence of $\mathbf{P}$, arranged alphabetically. For each conflict point $n\in\mathbf{N}$ with $\{\p_1,\p_2\} = \mathbf{P}_n$, set $z_{\p_1,\p_2,n} = 1$ if $\p_1$ precedes $\p_2$ in $\mathbf{P}^*$.

    Using \eqref{c:arrive-time}, we can now express \eqref{c:headway} as:
    \begin{equation}
        \firsthw{n} =
        \begin{cases}
            \toff{\p_2} - \toff{\p_1} + \ttravel{\p_2}{n} - \ttravel{\p_1}{n} - \dfrac{l}{v_f}, & \text{if } z_{\p_1,\p_2,n} = 1,\\[0.5em]
            \toff{\p_1} - \toff{\p_2} + \ttravel{\p_1}{n} - \ttravel{\p_2}{n} - \dfrac{l}{v_f}, & \text{otherwise},
        \end{cases}
        \quad \forall n\in\mathbf{N}.
        \label{c:headway-feas-proof}
    \end{equation}

    We choose $\toff{\p}$ to satisfy:
    \begin{equation}\label{feas-proof-subsys}
        \begin{aligned}
            \begin{cases}
                \toff{\p_2} - \toff{\p_1} = \dfrac{l}{v_f} + \tau_c + \ttravel{\p_1}{n} - \ttravel{\p_2}{n}, & \text{if } z_{\p_1,\p_2,n} = 1,\\[0.5em]
                \toff{\p_1} - \toff{\p_2} = \dfrac{l}{v_f} + \tau_c + \ttravel{\p_2}{n} - \ttravel{\p_1}{n}, & \text{otherwise},
            \end{cases}
            \quad \forall n\in\mathbf{N}.
        \end{aligned}
    \end{equation}

    To demonstrate the existence of solutions $\toff{\p}$ that satisfy (\ref{feas-proof-subsys}), we begin by noting that, in a general intersection, any two movements share at most one conflict point. In other words, for any pair $\p_1, \p_2 \in \mathbf{P}$, there exists a unique $n \in \mathbf{N}$ such that $\{\p_1, \p_2\} = \mathbf{P}_n$.
    Moreover, since the value of $z_{\p_1,\p_2,n}$ is determined solely by the order of $\p_1$ and $\p_2$ in the sequence $\mathbf{P}^*$, independent of $n$, and given that $\toff{\p}$ is a free continuous variable without an upper bound, it follows that suitable values of $\toff{\p} \ge 0$ can always be selected to satisfy subsystem (\ref{feas-proof-subsys}).

    Substituting \eqref{feas-proof-subsys} into \eqref{c:headway-feas-proof} yields $\firsthw{n} = \tau_c$. Constraints \eqref{c:bound1} are satisfied. Next, let:
    \begin{equation}
        C := \max\left\{2\tau_c + 2l/v_f,\ \max_{\p\in\mathbf{P}} \frac{1}{q_\p}\right\}.
    \end{equation}

    From \eqref{c:c-length-conf}, the second departure-arrival headway in a cycle is:
    \begin{equation}
        \secondhw{n} = C - (\platime{\p_1} + \platime{\p_2} + \firsthw{n})
            = C - \left(2l/v_f + \tau_c\right) \ge \tau_c,
    \end{equation}
    so that \eqref{c:bound2} is satisfied.

    For each $\p$ with $1/q_\p \le \Cmin$, define $d_\p := q_\p(r_\p + g_\p) - L_\p = q_\p C - 1 \ge 0$, since $C \ge \max_{\p\in\mathbf{P}} 1/q_\p$. Therefore, constraints \eqref{c:serv-rate} and \eqref{c:bound3} hold.
    Finally, since $C = \max\{2\tau_c + 2l/v_f,~\max_{\p\in\mathbf{P}} 1/q_\p\} \le \overline{C}$, constraints \eqref{c:bound2} also hold. \qed
\end{proof}

It is worth noting that constraints (\ref{c:headway}) are inherently non-linear due to the presence of logical conditions. These ``if–else" expressions can be systematically reformulated and subsequently linearized using the Reformulation–Linearization Technique (RLT) \citep{sherali2013reformulation, philip2024mixed}. RLT operates in two primary phases: reformulation and linearization. In the reformulation phase, redundant nonconvex constraints are introduced through the pairwise multiplication of existing linear or quadratic inequalities. During the linearization phase, each distinct product of variables is substituted with a newly defined continuous variable. As a result, constraints (\ref{c:headway}) can be equivalently represented in a linear form as follows:
\begin{remark}
    The constraints in (\ref{c:headway}) can be rearranged into the following equivalent non-linear form:
    \begin{equation}\label{c:headway_ref1}
        \firsthw{n} = z_{\p_1,\p_2,n} (\tarr{\p_2}{n} - \tarr{\p_1}{n} - \platime{\p_1}) + (1-z_{\p_1,\p_2,n})(\tarr{\p_1}{n} - \tarr{\p_2}{n} - \platime{\p_2}), \quad \forall n\in\mathbf{N},~\{\p_1,\p_2\}=\mathbf{P}_n
    \end{equation}

    To linearize constraints (\ref{c:headway_ref1}), we introduce the following continuous variables:
    \begin{equation}\label{RLT:vardefine}
        \begin{aligned}
            \linearvarXunder{\p_1}{\p_2}{n} := z_{\p_1,\p_2,n}\tarr{\p_1}{n}, &\quad \linearvarXover{\p_1}{\p_2}{n} := z_{\p_1,\p_2,n}\tarr{\p_2}{n},\\
            \linearvarYunder{\p_1}{\p_2}{n} := z_{\p_1,\p_2,n}\platime{\p_1}, &\quad \linearvarYover{\p_1}{\p_2}{n} := z_{\p_1,\p_2,n}\platime{\p_2}
        \end{aligned}
    \end{equation}

    Substituting these definitions, constraints (\ref{c:headway_ref1}) can be rewritten as:
    \begin{equation}\label{c:headway_ref2}
        \firsthw{n} = \tarr{\p_1}{n}-\tarr{\p_2}{n}-\platime{\p_2} + 2\overline{x}_{\p_1,\p_2,n} - 2\underline{x}_{\p_1,\p_2,n} - \underline{y}_{\p_1,\p_2,n} + \overline{y}_{\p_1,\p_2,n}, \quad \forall n\in\mathbf{N},~\{\p_1,\p_2\}=\mathbf{P}_n
    \end{equation}

    We then define valid lower and upper bounds for each continuous variable:
    \begin{equation}\label{RLT:bounds}
        \tarrLB{\p}{n}\le\tarr{\p}{n}\le\tarrUB{\p}{n}, \quad \platimeLB{\p}\le\platime{\p}\le\platimeUB{\p}
    \end{equation}
    
    Applying standard RLT principles, the non-linear terms in (\ref{RLT:vardefine}) can be linearized using the following constraints:
    \begin{equation}\label{RLT:1}
        \begin{alignedat}{2}
            \linearvarXunder{\p_1}{\p_2}{n} &\ge \tarrLB{\p_1}{n} z_{\p_1,\p_2,n}, \qquad
            &\linearvarXunder{\p_1}{\p_2}{n} &\le \tarrUB{\p_1}{n} z_{\p_1,\p_2,n}\\
            \linearvarXunder{\p_1}{\p_2}{n} &\ge \tarr{\p_1}{n} - \tarrUB{\p_1}{n} (1-z_{\p_1,\p_2,n}), \qquad
            &\linearvarXunder{\p_1}{\p_2}{n} &\le \tarr{\p_1}{n} - \tarrLB{\p_1}{n} (1-z_{\p_1,\p_2,n})
        \end{alignedat}
    \end{equation}
    \begin{equation}\label{RLT:2}
        \begin{alignedat}{2}
            \linearvarXover{\p_1}{\p_2}{n} &\ge \tarrLB{\p_2}{n} z_{\p_1,\p_2,n}, \qquad
            &\linearvarXover{\p_1}{\p_2}{n} &\le \tarrUB{\p_2}{n} z_{\p_1,\p_2,n}\\
            \linearvarXover{\p_1}{\p_2}{n} &\ge \tarr{\p_2}{n} - \tarrUB{\p_2}{n} (1-z_{\p_1,\p_2,n}), \qquad
            &\linearvarXover{\p_1}{\p_2}{n} &\le \tarr{\p_2}{n} - \tarrLB{\p_2}{n} (1-z_{\p_1,\p_2,n})
        \end{alignedat}
    \end{equation}
    \begin{equation}\label{RLT:3}
        \begin{alignedat}{2}
            \linearvarYunder{\p_1}{\p_2}{n} &\ge \platimeLB{\p_1} z_{\p_1,\p_2,n}, \qquad
            &\linearvarYunder{\p_1}{\p_2}{n} &\le \platimeUB{\p_1} z_{\p_1,\p_2,n}\\
            \linearvarYunder{\p_1}{\p_2}{n} &\ge \platime{\p_1} - \platimeUB{\p_1} (1-z_{\p_1,\p_2,n}), \qquad
            &\linearvarYunder{\p_1}{\p_2}{n} &\le \platime{\p_1} - \platimeLB{\p_1} (1-z_{\p_1,\p_2,n})
        \end{alignedat}
    \end{equation}
    \begin{equation}\label{RLT:4}
        \begin{alignedat}{2}
            \linearvarYover{\p_1}{\p_2}{n} &\ge \platimeLB{\p_2} z_{\p_1,\p_2,n}, \qquad
            &\linearvarYover{\p_1}{\p_2}{n} &\le \platimeUB{\p_2} z_{\p_1,\p_2,n}\\
            \linearvarYover{\p_1}{\p_2}{n} &\ge \platime{\p_2} - \platimeUB{\p_2} (1-z_{\p_1,\p_2,n}), \qquad
            &\linearvarYover{\p_1}{\p_2}{n} &\le \platime{\p_2} - \platimeLB{\p_2} (1-z_{\p_1,\p_2,n})
        \end{alignedat}
    \end{equation}

    Collectively, equations (\ref{c:headway_ref2})–(\ref{RLT:4}) provide a linear formulation that is equivalent to the original non-linear constraints (\ref{c:headway}).

    Next, we derive the bounds for $\tarr{\p}{n}$ and $\platime{\p}$. Since $L_\p \ge 1$, and by combining constraints (\ref{c:platoon-time}), the lower bound of $\platime{\p}$ can be set as $\platimeLB{\p} = l / v_f$. To establish its upper bound, we first note that the maximum platoon size occurs when 
    \begin{equation}
        L_\p = \left\lfloor \frac{C}{1/\qmax} \right\rfloor,
    \end{equation}
    and given that $C$ is upper bounded by $\overline{C}$, we obtain
    \begin{equation}
        \platimeUB{\p} = \left( \left\lfloor \frac{\overline{C}}{1/\qmax} \right\rfloor - 1 \right) \tau_f + \left\lfloor \frac{\overline{C}}{1/\qmax} \right\rfloor \frac{l}{v_f}.
    \end{equation}
    
    We impose a trivial lower bound on the arrival time as $\tarrLB{\p}{n} = 0$. To determine its upper bound, we first observe that in the optimal solution, $\toff{\p}$ satisfies
    \begin{equation}
        \toff{\p} \le (|\mathbf{P}| - 1) C \le (|\mathbf{P}| - 1) \overline{C},
    \end{equation}
    which implies that each movement cycle begins only after the first cycles of all other movements have concluded, and that cycles corresponding to different movements do not temporally overlap. 
    
    Moreover, since $r_\p = C - g_\p \le \overline{C} - g_\p$ and $g_\p$ is bounded below by $1/\qmax$ (because $L_\p = \qmax g_\p \ge 1$), it follows that
    \begin{equation}
        r_\p \le \overline{C} - \frac{1}{\qmax}.
    \end{equation}
    
    Combining the above with constraint (\ref{c:arrive-time}), the upper bound of $\tarr{\p}{n}$ can be expressed as
    \begin{equation}
        \tarrUB{\p}{n} = |\mathbf{P}| \, \overline{C} - \frac{1}{\qmax} + \ttravel{\p}{n}.
    \end{equation}
\end{remark}

\section{Numerical experiments}\label{sec:exps}
In this section, we evaluate the proposed CMAT framework against two representative benchmark methods under fully CAV traffic conditions. The first benchmark is the Rhythmic Control (RC) scheme \citep{chen2021rhythmic}, which enables vehicles to traverse intersections without stopping or conflicts by following a predefined rhythmic pattern. The second benchmark is Traditional Signal Control (TSC), which reflects current standard practice in traffic operations. For TSC, we apply the minimum cycle length formula \citep{webster1966traffic, urbanik2015signal} to compute signal timings aimed at minimizing the cycle length.
\begin{figure}[htbp]
    \centering
    \includegraphics[width=0.8\linewidth]{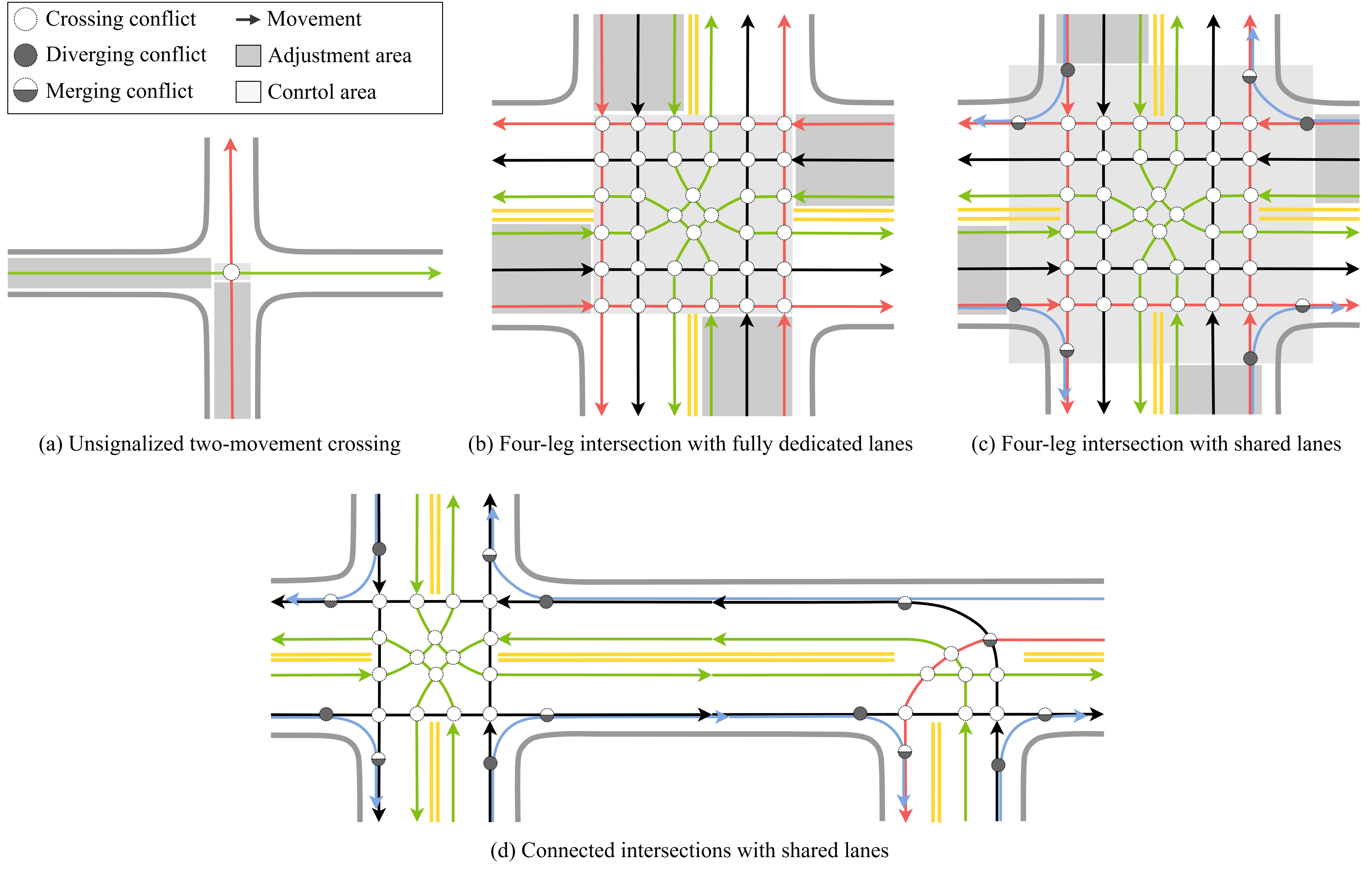}
    \caption{Layouts of intersections for test scenarios}
    \label{fig:test_scenarios}
    \captionsetup{justification=centering}
\end{figure}

To ensure a comprehensive assessment, four distinct traffic scenarios are considered, as illustrated in \autoref{fig:test_scenarios}. The first scenario involves an unsignalized two-movement crossing with a single conflict point. The second and third scenarios extend to four-leg intersections with multiple conflict points: one featuring fully dedicated lanes, and the other incorporating shared through/right-turn lanes.
The final scenario considers a connected coordination setting in which a four-leg intersection is linked to an asymmetric intersection, where the spacing between the two intersections is set to $L = 500$ m. When applying TSC to the connected intersections, signal timings at each intersection are optimized independently using the minimum cycle-length formula. A timing offset of $L / v_f$ seconds is then applied at the minor intersection to promote coordinated arrivals from the major intersection.

The key parameters are set as $v_f = 18$ m/s, $l = 4.5$ m, $\tau_f = 1$ s, $\tau_c = 2$ s, $\Cmin = 10$ s, $\lambda = 0.9$, $\ObjSwitchCoef = 0.05$, and $\overline{C} = 120$ s across all scenarios.

\subsection{Unsignalized single-conflict case}
In this unsignalized traffic scenario, we compare the performance of CMAT with that of RC. Two traffic demand patterns are examined. The first reflects a balanced demand condition, in which the movements exhibit similar traffic intensities. The second represents an imbalanced condition, where some movements carry much higher traffic volumes than others. To assess performance across a wide spectrum of traffic levels, we apply a scaling factor $\beta$, varied from $\underline{\beta}$ to $\overline{\beta}$ in increments of 0.1, to adjust the baseline demand vector proportionally.

In the balanced case, the demand vector is $\beta[1000, 1000]$ (veh/h/movement), where the entries correspond to the eastbound and northbound movements, respectively. In the imbalanced case, the demand vector is $\beta[1800, 100]$ (veh/h/movement), which imposes a markedly higher demand on the eastbound movement.
To examine how traffic intensity influences system performance, four measures are evaluated: (i) total throughput of the intersection, defined as the effective discharge rate (veh/h); (ii) average per-vehicle delay; (iii) the cycle length of CMAT; and (iv) the platoon sizes of CMAT. The results are shown in \autoref{fig:case1}.

\begin{figure}[htbp]
    \centering
    \includegraphics[width=1\linewidth]{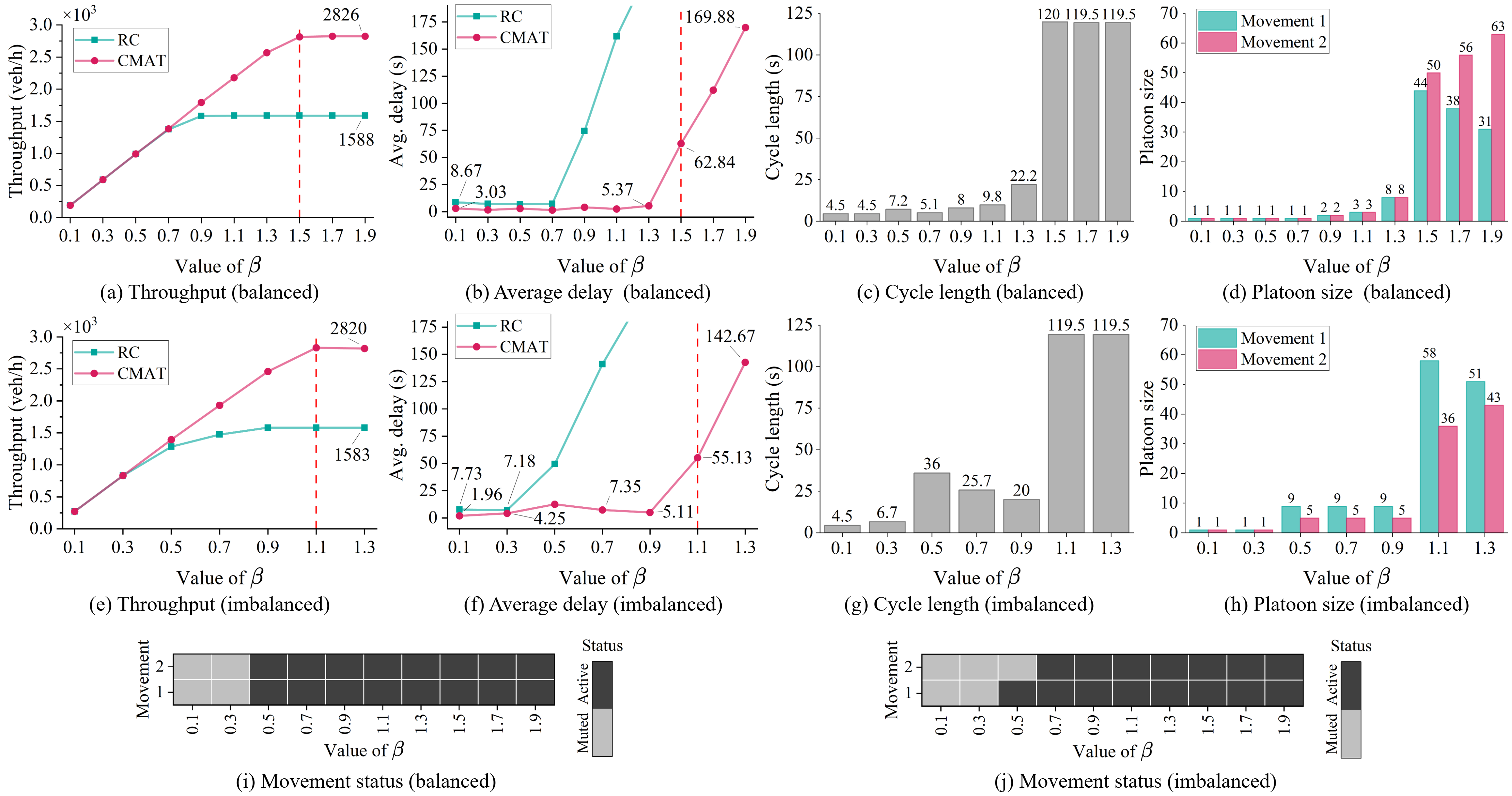}
    \caption{Results comparison in the unsignalized two-movement crossing (the red dashed line marks the point CMAT switches from model M1 to M2).}
    \label{fig:case1}
    \captionsetup{justification=centering}
\end{figure}

We observe that both throughput and delay exhibit a non-decreasing trend as $\beta$ increases under both demand patterns. In the low-flow setting, both methods are able to accommodate all arriving vehicles, while CMAT consistently yields lower delays than RC. As demand grows and the intersection becomes oversaturated, RC reaches its capacity $\frac{1}{\tau_c + l/v_f}$ earlier than CMAT in both scenarios. The improvement offered by CMAT is substantial: its capacity increases by approximately 80\%, from about \num{1500} veh/h to roughly \num{2800} veh/h. Meanwhile, CMAT also maintains significantly lower delays under oversaturation, as the limited capacity of RC leads to pronounced queue accumulation and consequently higher average delays.

We further note that the different demand patterns do not materially influence the relative capacity gains: the capacities of both methods remain essentially identical under balanced and imbalanced conditions. As discussed in \autoref{sec:exps}, throughput is governed by the average departure–arrival headway at the conflict point when the average speed is fixed. Under oversaturated conditions, regardless of the distribution of demand, RC alternates right-of-way between movements, causing its headway to converge to $\tau_c$. CMAT, in contrast, forms larger platoons on one movement as needed, driving its effective headway toward $\tau_f$. Therefore, the capacities of RC and CMAT follow $\frac{1}{\tau_c + l/v_f}$ and $\frac{1}{\tau_f + l/v_f}$, respectively.

An additional observation is that CMAT’s cycle length reaches its upper bound $\overline{C}$ once the capacity threshold is approached. This behavior arises from CMAT’s transition from model M1 to model M2. When arriving demand is too high for M1 to fully dissipate queues within a cycle, M1 becomes infeasible. CMAT then adopts M2, a relaxed formulation that removes the complete queue-clearance requirement and places greater emphasis on maximizing throughput. Thus, CMAT allocates larger platoon sizes and longer cycle lengths to boost overall throughput.

Overall, the marked increase in throughput (over $75\%$ improvement) and the substantial reduction in average delay over $50\%$ under non-saturated conditions demonstrate that CMAT can effectively enhance operational performance for a single conflict point setting.

\subsection{Four-leg intersection with fully dedicated lanes}
We extend the analysis to a more complex signalized multi-conflict scenario as shown in \autoref{fig:test_scenarios}(b), where each leg comprises two through lanes and one dedicated left-turn lane. The balanced demand pattern is specified as $\beta[1000, 1000, 1000]$ (veh/h/movement), and the imbalanced flow pattern is given by $\beta[1100, 1100, 500]$ (veh/h/movement). This demand vector is applied identically to each leg, with the first two entries corresponding to the through lanes and the third to the left-turn lane.
Since RC requires strict geometric alignment or redesign of the intersection layout, it is not included in this analysis. The comparison is therefore carried out between TSC and CMAT. The results are shown in \autoref{fig:case2}.

CMAT demonstrates strong potential in general multi-conflict settings, achieving more than a 50\% increase in capacity over TSC and exhibiting lower average delays. This improvement is partly due to CMAT’s ability to operate with shorter cycle lengths than TSC when demand is relatively low (prior to TSC reaching oversaturation). The reduction in cycle length arises from the greater flexibility of the ``phantom' micro-signal, which leverages CAV connectivity to eliminate inefficiencies inherent in TSC, such as all-red intervals, start-up lost time, and unnecessary deceleration.

Under both demand patterns, we also observe that platoon sizes tend to cluster near one when actual demand on certain movements is extremely low. This occurs because such movements become muted in the model, causing CMAT to assume the arrival of a single hypothetical vehicle per cycle, which slightly overestimates demand but is reasonable since platoon formation offers no advantage at low flow rates. Fixing platoon sizes to one under these conditions helps shorten the cycle and reduces delay. As $\beta$ increases and all movements become active, CMAT’s demand-responsive structure emerges clearly: platoon sizes expand and display greater variability across movements. The results further show that the imbalanced-demand case produces a wider dispersion of platoon sizes than the balanced case, with higher-volume movements receiving proportionally larger platoons.

An interesting phenomenon appears at $\beta=0.7$ under the imbalanced pattern, where TSC achieves slightly higher throughput and lower delay than CMAT. This outcome stems from CMAT’s underestimation of arrival demand, as some movements are still being \emph{muted}. In such cases, CMAT selects a substantially shorter cycle length than TSC. While a shorter cycle can reduce delay in principle, insufficient cycle length leads to the formation of queues, which in turn reduces total throughput. The waiting time of vehicles in these queues then further increases the overall average delay. Nevertheless, this phenomenon occurs only rarely when the arrival demand approaches TSC’s capacity, and it can be mitigated by slightly decreasing the threshold $\Cmin$ to re-activate movements that would otherwise remain \emph{muted}.

\begin{figure}[htbp]
    \centering
    \includegraphics[width=\linewidth]{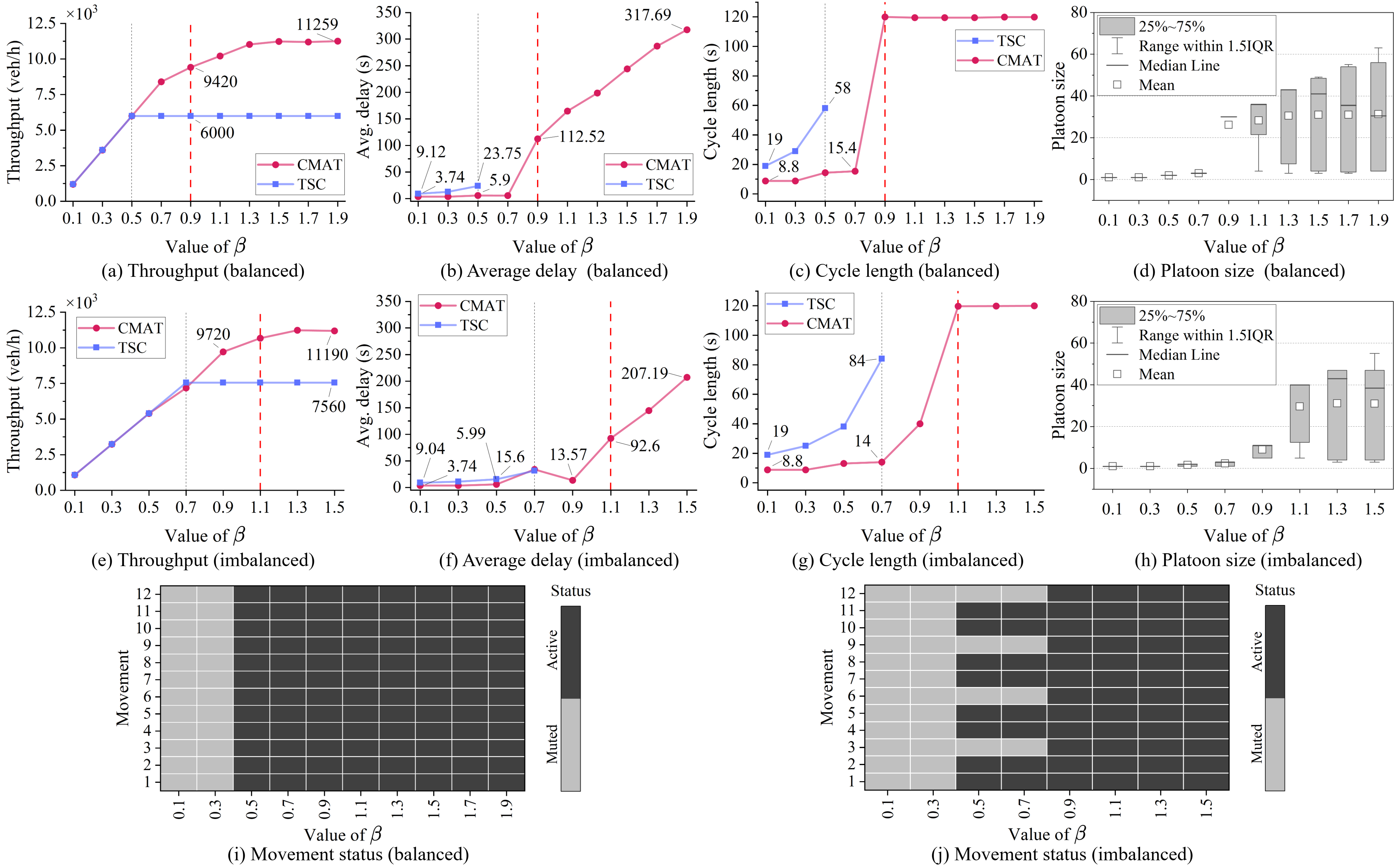}
    \caption{Results comparison in the four-leg intersection with fully dedicated lanes (the gray dashed line marks the point at which TSC is infeasible due to the oversaturation, while the red dashed line marks the point CMAT switches from model M1 to M2).}
    \label{fig:case2}
    \captionsetup{justification=centering}
\end{figure}

\subsection{Four-leg intersection with shared lanes}
We further extend the experiment to a more general four-leg intersection incorporating shared through/right lanes, as shown in \autoref{fig:test_scenarios}(c).
Following the previous setup, the arrival demand vectors for each leg are $\beta[500, 1000, 1000, 500]$ (veh/h/movement) and $\beta[450, 1800, 1800, 900]$ (veh/h/movement) for the balanced and imbalanced cases, respectively. The first entry denotes the right-turn movement, followed by the two through movements and the left-turn movement. The simulation results are provided in \autoref{fig:case3}.

\begin{figure}[htbp]
    \centering
    \includegraphics[width=1\linewidth]{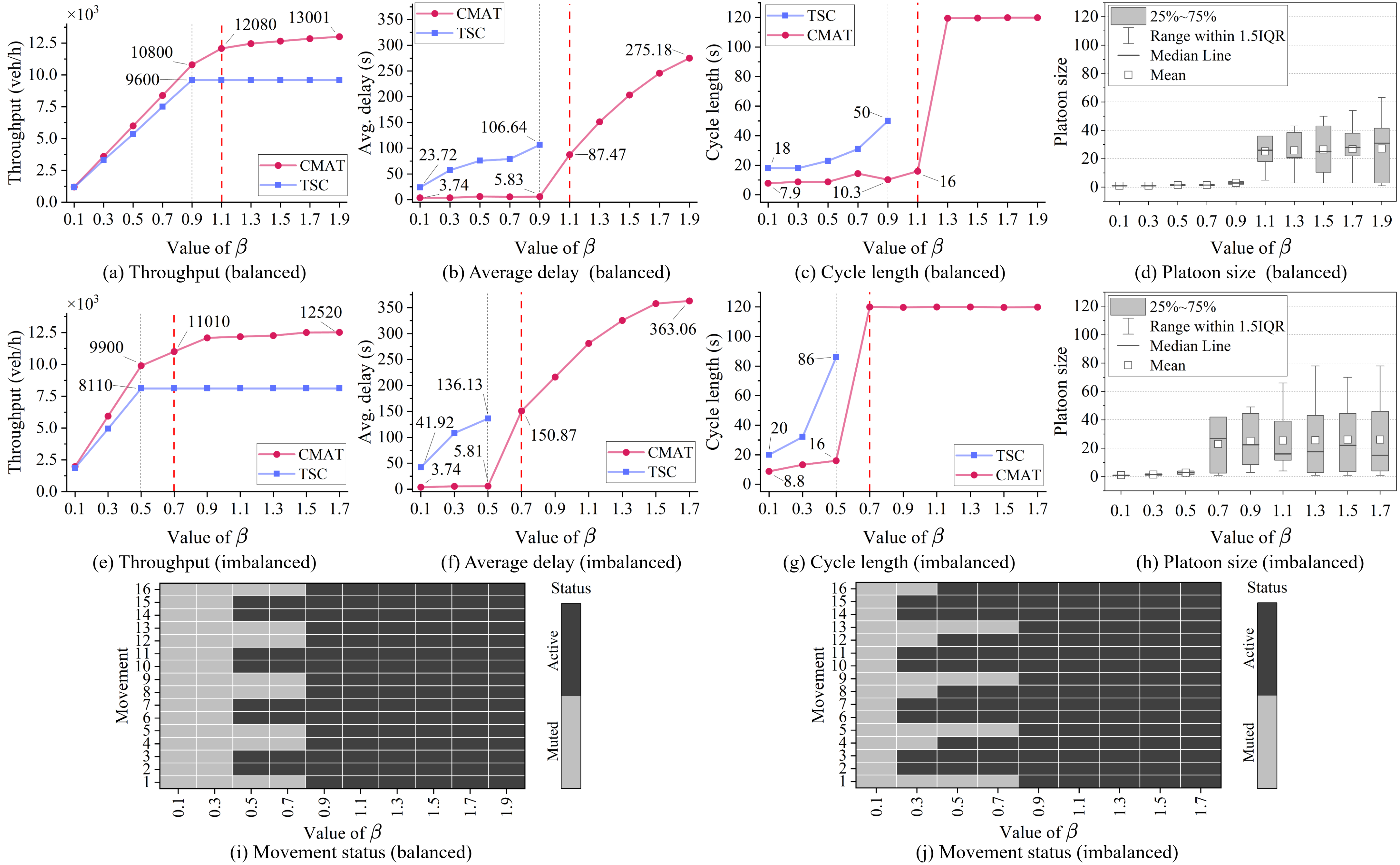}
    \caption{Results comparison in the four-leg intersection with shared lanes (the gray dashed line marks the point at which TSC is infeasible due to the oversaturation, while the red dashed line marks the point CMAT switches from model M1 to M2).}
    \label{fig:case3}
    \captionsetup{justification=centering}
\end{figure}

Consistent with the previous observations, CMAT continues to show the potential to increase intersection capacity, lower average delay, and operate with a shorter cycle than TSC. However, the capacity gain, which ranges from $35\%$ to $55\%$, is smaller than in the previous fully dedicated lane setting. This is mainly due to the increased capacity of TSC. The reason follows from the fundamental control structure: CMAT assigns right-of-way through a movement-based framework, in which the through and right-turn movements compete for use of the shared through/right-turn lane. These two movements are interdependent and maintain a diverging conflict relationship, alternatively owning the right-of-way. Therefore, adding the right-turn movement to the directed graph only slightly changes the overall conflict layout, which does not significantly influence the capacity improvement of CMAT. In contrast, TSC uses a phase-split timing structure, so the right-turn movement can share right-of-way with the through movement in the same phase. This increases the capacity of TSC relative to the fully dedicated lane case.

\subsection{Connected intersections with shared lanes}
To evaluate CMAT’s capability in managing multiple connected intersections at a network level, we extend the analysis to a scenario where a four-leg (major) intersection is connected to an asymmetric (minor) intersection, as illustrated in \autoref{fig:test_scenarios}(d). For consistency across simulations, we set the threshold cycle length to $\Cmin = 30$ seconds.

The performance of CMAT is compared against that of TSC under both balanced and imbalanced demand conditions, represented by the demand vectors $\beta[500, 1000, 500]$ (veh/h/movement) and $\beta[450, 1800, 900]$ (veh/h/movement), respectively. Each vector entry corresponds to the right-turn, through, and left-turn movements on each leg. For movements traversing both intersections, the flow rate is determined based on its classification at the major intersection. For instance, in the balanced scenario, the eastbound through movement at the minor intersection corresponds to a right-turn movement at the major intersection, thus receiving a flow rate of $500\beta$ veh/h. The simulation results are presented in \autoref{fig:case4}.

\begin{figure}[htbp]
    \centering
    \includegraphics[width=1\linewidth]{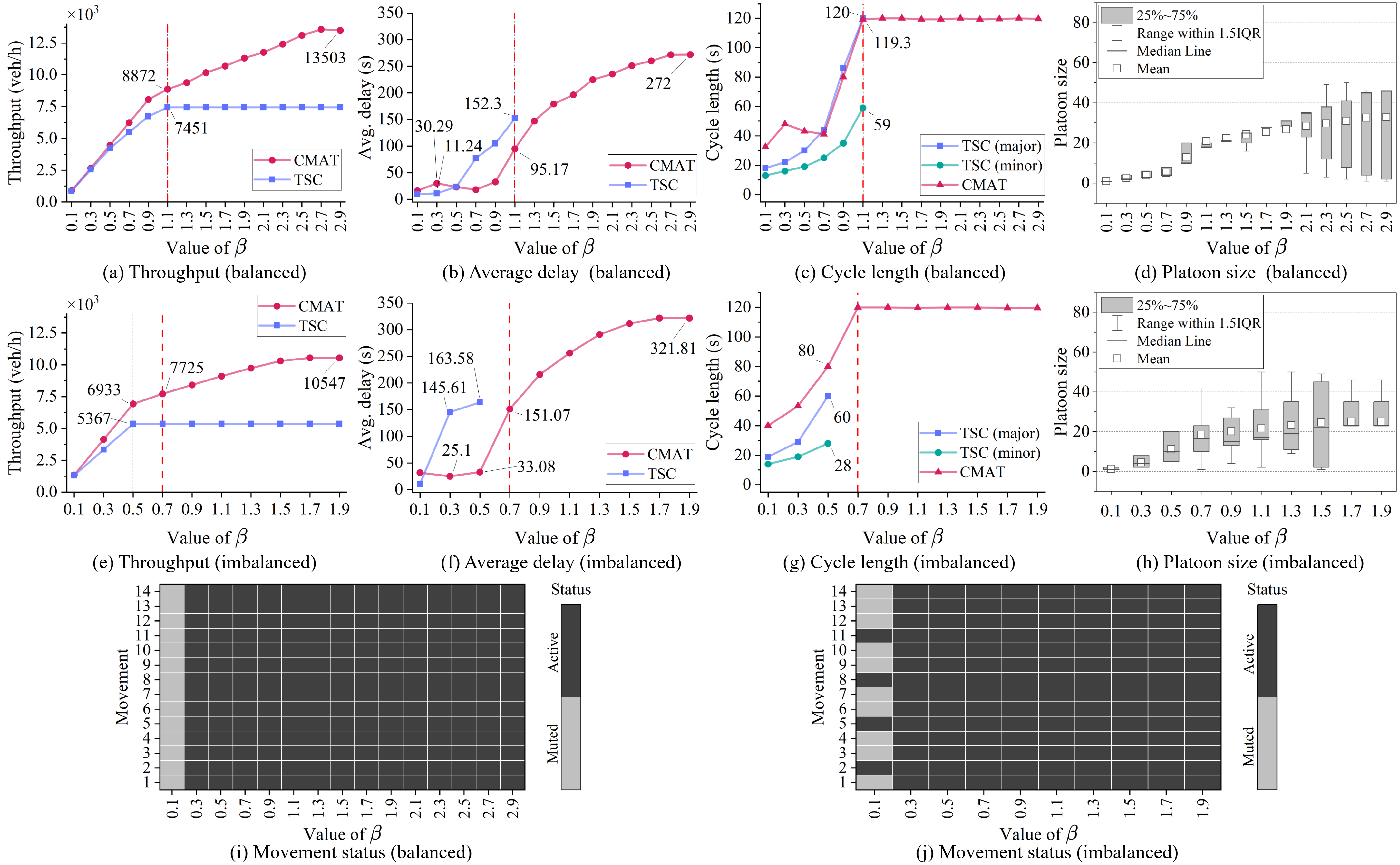}
    \caption{Results comparison in the connected intersections with shared lanes (the gray dashed line marks the point at which TSC is infeasible due to the oversaturation, while the red dashed line marks the point CMAT switches from model M1 to M2).}
    \label{fig:case4}
    \captionsetup{justification=centering}
\end{figure}

CMAT demonstrates strong effectiveness in coordinating connected intersections, achieving higher throughput and lower average delay than TSC in most cases. However, we observe instances where CMAT yields slightly lower throughput and higher delay than TSC (e.g., at $\beta=0.3$ in the balanced scenario and $\beta=0.1$ in the imbalanced scenario). As explained in the previous test case, this occurs when previously \emph{muted} movements become active once their arrival flow rate exceeds the threshold. For example, at $\beta=0.3$ under the balanced pattern, all movements are reactivated, leading to a sudden increase in cycle length. A similar transition occurs at $\beta=0.1$ under the imbalanced pattern, where the system shifts from partially muting movements to fully activating them.

CMAT demonstrates its effectiveness in coordinating connected intersections by achieving higher throughput and lower average delay compared to TSC. However, we observe a slightly smaller throughput and larger average delay in CMAT compared to TSC, e.g., $\beta=0.3$ in the balanced scenario and $\beta=0.1$ in the imbalanced case. As discussed in the previous test scenario, this happens when \emph{muted} movements begin to be active when their arrival flow rate reaches the threshold. In particular, when $\beta=0.3$ under the balanced demand pattern, all movements are re-activated and therefore create a sudden increase in the cycle length. Similarly, $\beta=0.1$ under the imbalanced scenario is the point transitioning from partially muting movements to fully activating all movements.

It is also worth noting that CMAT operates with longer cycle lengths in connected-intersection settings than in isolated intersections. This arises from the synchronization constraints: each isolated intersection’s cycle length must be a common multiple of the effective arrival headways, and coordinated operation requires the unified cycle to be a common multiple of the individual intersection cycles. As a result, the coordinated cycle cannot be shorter than any of the constituent cycles.

In summary, with more than an 80\% improvement in capacity and a 37.5\%–79.8\% reduction in average delay once TSC becomes oversaturated, CMAT exhibits strong capability for efficiently coordinating CAV traffic across connected intersections without requiring any manual offset tuning.

\section{Conclusions}\label{sec:conc}
This paper introduced Cyclic Modulation Control of Multi-Conflict Connected Automated Traffic (CMAT), a geometry-agnostic and demand-responsive framework designed to coordinate CAVs across arbitrary merging, diverging, and crossing conflict configurations. CMAT employs ``phantom'' micro-signals to induce the formation of CAV platoons in a demand-responsive manner, thereby shifting conflict resolution from individual vehicle slot negotiation to a platoon-based interaction paradigm, while minimizing the traffic delay. This transformation enables the average departure–arrival time headway across conflict points to approach the tighter car-following headway rather than the larger safety headways associated with crossing conflicts, which enhances overall capacity when necessary. CMAT introduces unified cycles that repeat over time at each conflict point, allowing platoons to traverse intersections at free-flow speed in a collision-free manner.

We formulate CMAT as a mixed-integer linear programming model based on a directed graph. CMAT is therefore geometry-agnostic, as the directed graph is abstracted from the physical intersection layout without imposing geometric constraints. Within this structure, essential geometric characteristics are preserved: nodes represent conflict points, and arcs denote the connections between these points and maintain the corresponding lengths of road segments. Theoretical analyses are provided to ensure that the proposed formulation admits feasible solutions under a wide spectrum of real traffic conditions, including oversaturation.

Extensive numerical experiments are carried out to validate the effectiveness of CMAT. The test scenarios range from a simple unsignalized single-conflict two-way crossing and a four-leg intersection with or without shared lanes, extending further to a connected-coordination setting where a four-leg intersection is linked with an asymmetric intersection. The state-of-the-art methods are included for comparison. Simulation results demonstrate that CMAT achieves the lowest average vehicle delay across all scenarios and varying traffic demand patterns. Comparatively, CMAT increases total capacity by more than 50\% and reduces average per-vehicle delay by over 50\% in most test settings.

Overall, CMAT establishes a scalable and theoretically grounded foundation for demand-responsive, high-throughput, and physically consistent CAV coordination. Future work will extend the framework to mixed traffic with high CAV penetration and examine robustness to sensing and communication imperfections. Additional directions include integrating pedestrians and cyclists, developing distributed and online solvers for network-scale deployment, and validating CMAT in digital-twin and field environments.

\bibliographystyle{elsarticle-harv}
\bibliography{cas-refs}






\newpage
\appendix






\end{document}